\documentclass[12pt, reqno]{amsart}
\usepackage{amsmath, amsthm, amscd, amsfonts, amssymb, graphicx, color}
\usepackage[bookmarksnumbered, colorlinks, plainpages]{hyperref}
\hypersetup{colorlinks=true,linkcolor=red, anchorcolor=green, citecolor=cyan, urlcolor=red, filecolor=magenta, pdftoolbar=true}

\newtheorem{theorem}{Theorem}[section]
\newtheorem{lemma}[theorem]{Lemma}

\newtheorem{corollary}[theorem]{Corollary}
\theoremstyle{definition}

\theoremstyle{remark}
\newtheorem{remark}[theorem]{\bfseries{Remark}}
\numberwithin{equation}{section}

\begin{document}
	
	\setcounter{page}{1}

\title [Numerical radius inequalities of operators and operator matrices]{Some improvements of numerical radius inequalities of operators and operator matrices} 

\author{Pintu Bhunia and Kallol Paul}

\address{(Bhunia) Department of Mathematics, Jadavpur University, Kolkata 700032, West Bengal, India}
\email{pintubhunia5206@gmail.com}

\address{(Paul) Department of Mathematics, Jadavpur University, Kolkata 700032, West Bengal, India}
\email{kalloldada@gmail.com}


\thanks{First author would like to thank UGC, Govt. of India for the financial support in the form of JRF. Second author would like to thank RUSA 2.0, Jadavpur University for the partial support.}


\subjclass[2010]{Primary 47A12, Secondary 47A05, 46C05.}
\keywords{Numerical radius; B-numerical radius; A-adjoint operator; Operator matrix}

\maketitle
\begin{abstract}
We obtain upper bounds for the numerical radius of a product of Hilbert space operators which improve on the existing upper bounds. We generalize the numerical radius inequalities of $n\times n$ operator matrices by using non-negative continuous functions on $[0,\infty)$. 
We also obtain some upper and lower bounds for the $B$-numerical radius of operator matrices, where $B$ is the diagonal operator matrix whose each diagonal entry is a positive operator $A.$ We show that these bounds generalize and improve on the existing bounds.
\end{abstract}

\section{Introduction}
\noindent The purpose of the present article is to study the inequalities for the numerical radius of a product of two bounded linear operators defined on a complex Hilbert space and obtain inequalities for the bounds of the numerical radius of $ n \times n $ operator matrices. We also study the inequalities for the B-numerical radius of operator matrices, where $B$ is a diagonal operator matrix whose each diagonal entry is a positive operator A defined on a complex Hilbert space.  For this we need   the following notations and terminologies.\\
Let $\mathbb{H}$ be a nontrivial complex Hilbert space with the usual inner product $\langle .,. \rangle$, and $\|.\|$ be the norm induced from $\langle .,. \rangle$.
Let $B(\mathbb{H})$ denote the $C^*$-algebra of all bounded linear operators on $\mathbb{H}.$ For $T\in B(\mathbb{H})$, let $\|T\|$, $W(T)$ and $w(T)$ denote the operator norm, the numerical range and the numerical radius of $T$ respectively, defined as follows: 
\begin{eqnarray*}
\|T\|&=& \sup\{\|Tx\| :  x\in \mathbb{H}, \|x\|=1\},\\
W(T) &=& \{\langle Tx,x\rangle: x\in \mathbb{H}, \|x\|=1\},\\
w(T) &=& \sup\{|\lambda|: \lambda \in W(T)\}.
\end{eqnarray*}
The spectral radius of $T$, denoted as $ r(T)$, is defined as the radius of the smallest circle with centre at the origin that contains the spectrum.  It is well known that the closure of the numerical range contains the spectrum and so $ r(T) \leq w(T).$ 
It is easy to verify that $w(T)$ is a norm on $B(\mathbb{H})$ and is equivalent to the operator norm, and satisfies the following inequality
\[\frac{1}{2}\|T\|\leq w(T)\leq \|T\|.\]
Various numerical radius inequalities improving this inequality have been given in \cite{AK,A,BBP,BBP3,BBP1,BBP2,BPN1,OFK,HD,K,PB,PB2,SMY,Y}. We reserve the letters $I$ and $O$ for the identity operator and the zero operator defined on $\mathbb{H}$, respectively. A self-adjoint operator $A\in B(\mathbb{H})$ is called positive  if $\langle Ax,x \rangle \geq 0$ for all $x\in \mathbb{H}$ and  is called strictly positive if $\langle Ax,x \rangle > 0$ for all non-zero  $x\in \mathbb{H}$. For a positive (strictly positive) operator $A$ we write $A\geq 0$ $(A>0)$. In this article, we use the letter $A$ for a positive operator on $\mathbb{H}$. Clearly, a positive operator $A$ induces a positive semi-definite sesquilinear form $\langle .,. \rangle _A : \mathbb{H} \times \mathbb{H} \rightarrow \mathbb{C}$ defined as $\langle x,y \rangle _A=\langle Ax,y \rangle$, for all $x,y\in \mathbb{H}$. Let $\|.\|_A$ denote the semi-norm on $\mathbb{H}$ induced from the  sesquilinear form $\langle .,. \rangle_A,$ i.e., $\|x\|_A=\sqrt{\langle x,x \rangle_A}$, for all $x\in \mathbb{H}.$ It is easy to verify that $\|.\|_A$ is a norm on $\mathbb{H}$ if and only if $A$ is a strictly positive operator. For $T\in B(\mathbb{H})$, the A-operator semi-norm of $T$, denoted as $\|T\|_A$,  is defined as \[\|T\|_A=\sup_{x\in  \overline{R(A)},x\neq 0}\frac{\|Tx\|_A}{\|x\|_A}.\]
For $T\in B(\mathbb{H}),$ the A-minimum norm of $T$, denoted as $c_A(T)$, is defined as 
\[c_A(T)=\inf_{x\in  \overline{R(A)},x\neq 0}\frac{\|Tx\|_A}{\|x\|_A}.\]
The generalization of the numerical range, known as the A-numerical range (see \cite{BFA}), and denoted as $W_A(T)$, is defined as 
\[W_A(T) = \{\langle Tx,x\rangle_A: x\in \mathbb{H}, \|x\|_A=1\}.\] 
  The A-numerical radius $w_A(T)$ and the A-Crawford number $m_A(T)$ of $T$ are defined  as 
\[ w_A(T)=\sup\{|\lambda|: \lambda \in W_A(T)\}, \]
\[m_A(T)=\inf\{|\lambda|: \lambda \in W_A(T)\}.\]
For a given $T\in B(\mathbb{H})$, if there exists $c>0$ such that $\|Tx\|_A\leq c\|x\|_A$ for all $x\in \mathbb{H}$ then the A-numerical radius $w_A(T)<+\infty$  and satisfies the following inequality 
\[\frac{1}{2}\|T\|_A \leq w_A(T)\leq \|T\|_A.\]
In \cite{Z}, Zamani studied the A-numerical radius inequalities for the semi-Hilbertian space operators. In \cite{BPN}, we have also studied the $B$-numerical radius inequalities of $2\times 2$ operator matrices, where $B$ is a $2\times 2$ diagonal operator matrix whose each diagonal entry is a positive operator $A$. Let $B^A(\mathbb{H})=\{ T\in B(\mathbb{H}): \|T\|_A < +\infty\}.$ It is well-known that $B^A(\mathbb{H})$ is not generally a sub-algebra of $B(\mathbb{H})$.
For $T \in B(\mathbb{H})$, an operator $R \in B(\mathbb{H})$ is called an $A$-adjoint of $T$ if for all $x,y \in \mathbb{H}$, $\langle Tx,y \rangle_A=\langle x,Ry \rangle_A$, i.e., $AR=T^*A$, where $T^*$ is the adjoint of $T$. Neither the existence nor the uniqueness of A-adjoint of $T$ holds true in general. Let $B_A(\mathbb{H})$ denote the collection of all operators on $\mathbb{H}$ which admit A-adjoints. It is well-known that $B_A(\mathbb{H})$ is a sub-algebra of $B(\mathbb{H})$. Moreover, the following inclusions hold \[B_A(\mathbb{H}) \subseteq B^A(\mathbb{H})\subseteq B(\mathbb{H}),\] the equality holds  if $A$ is injective and $A$ has a closed range. For $T \in B(\mathbb{H})$, A-adjoint operator of $T$ is written as $T^{\sharp_A}$. It is useful to note that if $T \in B_A(\mathbb{H})$ then $AT^{\sharp_A}=T^*A$. An operator $T\in B_A(\mathbb{H})$ is said to be an A-self-adjoint operator if $AT$ is self-adjoint, i.e., $AT=T^*A.$ An operator $U\in  B_A(\mathbb{H})$ is said to be $A$-unitary if $\|Ux\|_A=\|x\|_A$ and $\|U^{\sharp_A}x\|_A=\|x\|_A$, for all $x\in \mathbb{H}$. We note that if $T\in B_A(\mathbb{H})$ then $T^{\sharp_A}\in B_A(\mathbb{H})$ and $(T^{\sharp_A})^{\sharp_A}=PTP$, where $P$ is an orthogonal projection onto $\overline{R(A)}$,  $R(A)$ being the range of the operator $A$. Clearly $T^{\sharp_A}T$, $TT^{\sharp_A}$ are A-self-adjoint and A-positive operators satisfying 
$\|T^{\sharp_A}T\|_A =\|TT^{\sharp_A}\|_A =\|T\|^2_A =\|T^{\sharp_A}\|^2_A.$
Also for $T,S \in  B_A(\mathbb{H})$,  $(TS)^{\sharp_A}=S^{\sharp_A}T^{\sharp_A}$, $\|TS\|_A\leq \|T\|_A\|S\|_A$ and $\|Tx\|_A\leq \|T\|_A\|x\|_A$ for all $x\in \mathbb{H}.$ For more information we refer the reader to \cite{ ACG2, AS, BFA}. For $T\in B_A(\mathbb{H})$, we write $\textit{Re}_A(T)=\frac{1}{2}(T+T^{\sharp_A})$ and $\textit{Im}_A(T)=\frac{1}{2i}(T-T^{\sharp_A})$. \\

\noindent In section $2$, we obtain new upper bounds for  the numerical radius of a product of two bounded linear operators defined on a complex Hilbert space $\mathbb{H}$ and show that these bounds improve on the existing bounds. In section $3$, we obtain some upper bounds for the numerical radius of $n\times n$ operator matrices by using two non-negative continuous functions on $[0,\infty)$, which improve on the existing upper bounds. The results of sections $2$ and $3$ are improvements and generalizations as well as corrections of the results proved by Alomari in \cite{A}. In section $4$, we obtain new  bounds for the $B$-numerical radius of $n\times n$ operator matrices, where $B$ is an $n\times n$ diagonal operator matrix whose each diagonal entry is a positive operator $A$ defined on $\mathbb{H}$. The section 4 is an extension of the  recent work \cite{BPN} on the $B$-numerical radius.

\section{\textbf{Bounds for the numerical radius of a product of two operators}}  
\noindent In this section, we obtain  upper bounds for the numerical radius of a product of two operators in $B(\mathbb{H}),$ which improve on the existing bounds in \cite{A}. To obtain these bounds, we need the following lemmas. First lemma is known as Power-Young inequality and the second one is known as McCarthy inequality.

\begin{lemma}$($\cite{SMY}$)$\label{lemma-1}
Let $a,b\geq 0$ and $\alpha, \beta>1$ be such that $\frac{1}{\alpha}+\frac{1}{\beta}=1.$ Then \[ ab\leq \frac{1}{\alpha}a^{\alpha}+\frac{1}{\beta}b^{\beta}.\] 
\end{lemma}

\begin{lemma}$($\cite{K88}$)$.\label{lemma-2}
Let $A\in B(\mathbb{H})$ be positive operator, i.e., $A\geq 0.$ Then \[\langle Ax,x\rangle^p\leq \langle A^px,x\rangle,\] for all $p\geq 1$ and for all $x \in \mathbb{H}$ with $\|x\|=1.$
\end{lemma}

\begin{lemma}$(${\cite[Th. 5]{K88}}$)$.\label{lemma-3}
Let $X,Y\in B(\mathbb{H})$ be such that $|X|Y=Y^*|X|,$ where $|X|=(X^{*}X)^{\frac{1}{2}}$. Let $f$ and $g$ be two non-negative continuous functions on $[0,\infty)$ such that $f(t)g(t)=t$, for all $t\in [0,\infty).$ Then \[|\langle XYx,y\rangle|\leq r(Y) \|f(|X|)x\| \|g(|X^*|)y\|, \] for all $x,y \in \mathbb{H}. $
\end{lemma}

The  following inequality for the numerical radius of $2 \times 2$ operator matrices follows from   \cite[Th. 2.5]{BBP3}.
\begin{lemma}\label{lemma-4} 
Let $X,Y \in B(\mathbb{H})$ and $T=\left(\begin{array}{cc}
O&X\\
Y&O
\end{array}\right).$ Then
\[w^2(T)\leq \frac{1}{4}\Big \| |X|^2+|Y^*|^2\Big\|+\frac{1}{2}w(YX).\]
In particular, 
\[w^2(X)\leq \frac{1}{4}\Big \||X|^2+|X^*|^2\Big\|+\frac{1}{2}w(X^2).\]
\end{lemma}

We are now in a position to obtain the desired bound for the numerical radius of a product of two operators.

\begin{theorem}\label{theorem-1}
Let $X,Y\in B(\mathbb{H})$ be such that $|X|Y=Y^*|X|.$ Let $f$ and $g$ be two non-negative continuous functions on $[0,\infty)$ such that $f(t)g(t)=t$, for all $t\in [0,\infty).$ Then 
\begin{eqnarray*}
w^p(XY)&\leq& 2r^p(Y)~ w \left(\begin{array}{cc}
O& \frac{1}{\alpha}f^{p\alpha}(|X|)\\
\frac{1}{\beta}g^{p\beta}(|X^*|) & O
\end{array}\right)\\
&\leq& r^p(Y) \left[\Big\|  \frac{1}{\alpha^2}|f^{p\alpha}(|X|)|^2+\frac{1}{\beta^2}|g^{p\beta}(|X^*|)|^2 \Big\|+\frac{2}{\alpha \beta}\left\| g^{p\beta}(|X^*|) f^{p\alpha}(|X|)\right\| \right]^{\frac{1}{2}},
\end{eqnarray*}
where $p\geq 1$ and  $\alpha, \beta >1$ be such that $\frac{1}{\alpha}+\frac{1}{\beta}=1$, $p\alpha\geq 2,$ $p\beta \geq 2.$
\end{theorem}

\begin{proof}
Let $x\in \mathbb{H}$ with $\|x\|=1.$ Then from Lemma \ref{lemma-3}, we get  
\begin{eqnarray*}
|\langle XYx,x\rangle|&\leq&  r(Y) \|f(|X|)x\| \|g(|X^*|)x\|\\
\Rightarrow |\langle XYx,x\rangle|&\leq& r(Y)\left[ \frac{1}{\alpha} \|f(|X|)x\|^{\alpha} +\frac{1}{\beta}\|g(|X^*|)\|^{\beta}\right],~~ \textit{using Lemma }\ref{lemma-1}\\
\Rightarrow |\langle XYx,x\rangle|&\leq& r(Y) \left[\frac{1}{\alpha} \langle f^2(|X|)x,x \rangle^{\frac{\alpha}{2}}+ \frac{1}{\beta} \langle g^2(|X^*|)x,x \rangle^{\frac{\beta}{2}} \right]\\
\Rightarrow |\langle XYx,x\rangle|^p &\leq&  r^p(Y) \left[\frac{1}{\alpha} \langle f^2(|X|)x,x \rangle^{\frac{\alpha}{2}}+ \frac{1}{\beta} \langle g^2(|X^*|)x,x \rangle^{\frac{\beta}{2}} \right]^p\\
\Rightarrow |\langle XYx,x\rangle|^p &\leq&  r^p(Y) \left[\frac{1}{\alpha} \langle f^2(|X|)x,x \rangle^{\frac{p\alpha}{2}}+ \frac{1}{\beta} \langle g^2(|X^*|)x,x \rangle^{\frac{p\beta}{2}} \right], ~~\textit{using convexity of} ~~t^p \\
\Rightarrow |\langle XYx,x\rangle|^p&\leq&  r^p(Y) \left[\frac{1}{\alpha} \langle f^{p\alpha}(|X|)x,x \rangle + \frac{1}{\beta} \langle g^{p\beta}(|X^*|)x,x \rangle \right], ~~\textit{using Lemma }\ref{lemma-2}\\
\Rightarrow |\langle XYx,x\rangle|^p&\leq & r^p(Y) \left[ \langle \left(\frac{1}{\alpha}f^{p\alpha}(|X|)+\frac{1}{\beta}g^{p\beta}(|X^*|)\right)x,x\rangle \right]\\
\Rightarrow |\langle XYx,x\rangle|^p &\leq & r^p(Y)w \left(\frac{1}{\alpha}f^{p\alpha}(|X|)+\frac{1}{\beta}g^{p\beta}(|X^*|)\right)
 \end{eqnarray*}
\begin{eqnarray*}
\Rightarrow |\langle XYx,x\rangle|^p&\leq & r^p(Y)  \left \|\frac{1}{\alpha}f^{p\alpha}(|X|)+\frac{1}{\beta}g^{p\beta}(|X^*|)\right \|\\
\Rightarrow |\langle XYx,x\rangle|^p&\leq & 2r^p(Y) w \left(\begin{array}{cc}
O& \frac{1}{\alpha}f^{p\alpha}(|X|)\\
\frac{1}{\beta}g^{p\beta}(|X^*|) & O
\end{array}\right),  ~~\textit{ using  ~~\cite[Cor. 3]{AK}}\\
\Rightarrow |\langle XYx,x\rangle|^p &\leq& r^p(Y) \Bigg[\Big\|  \frac{1}{\alpha^2}|f^{p\alpha}(|X|)|^2+\frac{1}{\beta^2}|g^{p\beta}(|X^*|)|^2 \Big\| \\ 
&& \hspace{2cm} +\frac{2}{\alpha \beta}\left\| g^{p\beta}(|X^*|) f^{p\alpha}(|X|)\right\| \Bigg]^{\frac{1}{2}}, ~~\textit{using Lemma} ~~\ref{lemma-4}.
\end{eqnarray*}
Taking supremum over $x\in \mathbb{H}, \|x\|=1$, we get the desired bound and this completes the proof.
\end{proof}

In particular, if we take $\alpha=\beta=2$ and $p=1$ in Theorem \ref{theorem-1}, then we get the following inequality.

\begin{corollary}\label{corollary-1}
Let $X,Y \in B(\mathbb{H})$ be such that $|X|Y=Y^*|X|.$ Then 
\[w(XY)\leq r(Y) ~~ w\left(\begin{array}{cc}
O& f^2(|X|)\\
g^2(|X^*|)&O
\end{array}\right),\] where $f$ and $g$ are two non-negative continuous functions on $[0,\infty)$ such that  $f(t)g(t)=t, $ for all $t \in [0,\infty)$. \\
\end{corollary}

Considering  $f(t)=g(t)=t^{\frac{1}{2}}$ in  Corollary \ref{corollary-1} and noting that $f^2(|X|)=|X|$, $g^2(|X^*|)=|X^*|$ and $r(Y)\leq w(Y)$, we get the following inequality by using Lemma \ref{lemma-4}.

\begin{corollary}\label{corollary-2}
Let $X,Y \in B(\mathbb{H})$ be such that $|X|Y=Y^*|X|.$ Then 
\[w(XY)\leq \frac{1}{4}\Big( \left\| |X|^2+|X^*|^2\right\|+2\left\|X^2\right\|\Big)^{\frac{1}{2}}~~ \Big( \left\| |Y|^2+|Y^*|^2\right\|+2\left\|Y^2\right\|\Big)^{\frac{1}{2}}.\] 
\end{corollary}

\begin{remark}\label{remark-p}
For any $T\in B(\mathbb{H})$, we have $\|T^2\|=\|T^2\|^{\frac{1}{2}}\|T^2\|^{\frac{1}{2}}\leq \|T\|\|T^2\|^{\frac{1}{2}}.$ Therefore, using \cite[Lemma~~ 7]{FK3} we have
\begin{eqnarray*}
\left \| |T|^2+|T^*|^2 \right\|+2\left\|T^2\right\| &\leq & \|T^2\|+\|T\|^2+2\left\|T^2\right\| \\
&\leq & \|T^2\|+\|T\|^2+2\|T\|\|T^2\|^{\frac{1}{2}}\\
&=& (\|T\|+\|T^2\|^{\frac{1}{2}})^2.
\end{eqnarray*}
Thus the bound obtained in Corollary \ref{corollary-2} improves on the bound obtained by Alomari in \cite[Cor. 3.2]{A}.  
\end{remark}

To prove next theorem, we need Buzano's inequality, which is a generalization of Cauchy-Schwarz inequality.

\begin{lemma}{$($\cite{B}$)$.} \label{lemma-5}
Let $a,b,x\in \mathbb{H}.$ Then
\[|\langle a,x\rangle \langle x,b\rangle|\leq \frac{\|a\|\|b\|+|\langle a,b\rangle|}{2}\|x\|^2.\]
\end{lemma}

\begin{theorem}\label{theorem-2}
Let $X,Y\in B(\mathbb{H})$ be such that $XY=YX$ and $|X^2|Y^2=(Y^2)^*|X^2|$. Let $f$ and $g$ be two non-negative continuous functions on $[0,\infty)$ such that $f(t)g(t)=t,$ for all $ t\in [0,\infty).$ Then
\begin{eqnarray*}
w^{2p}(XY) &\leq & \frac{\|XY\|^{2p}}{2}+r^p(Y^2) ~~ w\left(\begin{array}{cc}
O& \frac{1}{\alpha}f^{p\alpha}(|X^2|)\\
\frac{1}{\beta}g^{p\beta}(|(X^*)^2|) & O
\end{array}\right)\\
&\leq& \frac{\|XY\|^{2p}}{2}+\frac{1}{2}r^p(Y^2) \Bigg[\Big\|  \frac{1}{\alpha^2}|f^{p\alpha}(|X^2|)|^2 +\frac{1}{\beta^2}|g^{p\beta}(|(X^*)^2|)|^2 \Big\| \\ 
&& \hspace{4.7cm}+\frac{2}{\alpha \beta}\left\| g^{p\beta}(|(X^*)^2|) f^{p\alpha}(|X^2|)\right\| \Bigg]^{\frac{1}{2}},
\end{eqnarray*}
where $p\geq 1$ and $\alpha,\beta >1$ be such that $\frac{1}{\alpha}+\frac{1}{\beta}=1$, $p\alpha\geq 2$, $p\beta \geq 2.$
\end{theorem}

\begin{proof}
Let $x\in \mathbb{H}$ be such that $\|x\|=1$. Taking $a=XYx$ and $b=Y^*X^*x$ in Lemma \ref{lemma-5}, we get
\begin{eqnarray*}
|\langle XYx,x\rangle|^2&\leq&\frac{1}{2}\left(\|XYx\|\|Y^*X^*x\|+|\langle (XY)^2x,x\rangle|\right)\\
&\leq& \frac{\|XY\|^2}{2}+\frac{|\langle X^2Y^2x,x\rangle|}{2}, ~~\textit{as}~~XY=YX.
\end{eqnarray*}
Using convexity of $t^p$, we get
\[|\langle XYx,x\rangle|^{2p}\leq \frac{\|XY\|^{2p}}{2}+\frac{|\langle X^2Y^2x,x\rangle|^p}{2}.\]
Proceeding similarly as in the proof of Theorem \ref{theorem-1} and noting that $|X^2|Y^2=(Y^*)^2|X^2|$, we get
\begin{eqnarray*}
|\langle X^2Y^2x,x\rangle|^{p} &\leq & 2r^p(Y^2) ~~ w\left(\begin{array}{cc}
O& \frac{1}{\alpha}f^{p\alpha}(|X^2|)\\
\frac{1}{\beta}g^{p\beta}(|(X^*)^2|) & O
\end{array}\right)\\
&\leq& r^p(Y^2) \Bigg[\Big\|  \frac{1}{\alpha^2}|f^{p\alpha}(|X^2|)|^2 +\frac{1}{\beta^2}|g^{p\beta}(|(X^*)^2|)|^2 \Big\| \\ 
&& \hspace{4cm}+\frac{2}{\alpha \beta}\left\| g^{p\beta}(|(X^*)^2|) f^{p\alpha}(|X^2|)\right\| \Bigg]^{\frac{1}{2}}.
\end{eqnarray*}
Therefore,
\begin{eqnarray*}
|\langle XYx,x\rangle|^{2p} &\leq & \frac{\|XY\|^{2p}}{2}+r^p(Y^2) ~~ w\left(\begin{array}{cc}
O& \frac{1}{\alpha}f^{p\alpha}(|X^2|)\\
\frac{1}{\beta}g^{p\beta}(|(X^*)^2|) & O
\end{array}\right)\\
&\leq& \frac{\|XY\|^{2p}}{2}+\frac{1}{2}r^p(Y^2) \Bigg[\Big\|  \frac{1}{\alpha^2}|f^{p\alpha}(|X^2|)|^2 +\frac{1}{\beta^2}|g^{p\beta}(|(X^*)^2|)|^2 \Big\| \\ 
&& \hspace{4.3cm}+\frac{2}{\alpha \beta}\left\| g^{p\beta}(|(X^*)^2|) f^{p\alpha}(|X^2|)\right\| \Bigg]^{\frac{1}{2}}.
\end{eqnarray*}
Taking supremum over $x\in \mathbb{H}, \|x\|=1$, we get the required inequality and this completes the proof of the theorem.
\end{proof}

In particular, if we take  $\alpha=\beta=2$ and $p=1$ in Theorem \ref{theorem-2}, then we get the following inequality:
\begin{corollary}\label{corollary-3}
Let $X,Y \in B(\mathbb{H})$ be such that $XY=YX$ and $|X^2|Y^2=(Y^*)^2|X^2|.$ Then 
\[w^2(XY)\leq \frac{\|XY\|^2}{2}+\frac{1}{2} r(Y^2) ~~ w\left(\begin{array}{cc}
O& f^2(|X^2|)\\
g^2(|(X^*)^2|)&O
\end{array}\right),\] where $f$ and $g$ are two non-negative continuous functions on $[0,\infty)$ such that $f(t)g(t)=t,$ for all  $ t\in [0,\infty).$ 
\end{corollary}

Considering  $f(t)=g(t)=t^{\frac{1}{2}}$ in Corollary \ref{corollary-3}  and noting that  $f^2(|X^2|)=|X^2|$, $g^2(|(X^*)^2|)=|(X^*)^2|$ and $r(Y^2)\leq w(Y^2)$ we get the following inequality by using Lemma \ref{lemma-4}.

\begin{corollary}\label{corollary-4}
Let $X,Y \in B(\mathbb{H})$ be such that $XY=YX$ and $|X^2|Y^2=(Y^*)^2|X^2|.$ Then 
\begin{eqnarray*}
w^2(XY)&\leq& \frac{\|XY\|^2}{2}\\
&+&\frac{1}{8}\Big( \left\| |X^2|^2+|(X^*)^2|^2\right\|+2\left\|X^4\right\|\Big)^{\frac{1}{2}} \Big( \left\| |Y^2|^2+|(Y^*)^2|^2\right\|+2\left\|Y^4\right\|\Big)^{\frac{1}{2}}.
\end{eqnarray*}
\end{corollary}

\begin{remark}
Alomari \cite[Cor. 3.3]{A} obtained an inequality, which seems to have some typos. This inequality should be of the form
\[w^2(XY)\leq \frac{1}{2}\|XY\|^2+\frac{1}{8}\left(\|Y^2\|+\|Y^4\|^{\frac{1}{2}}\right) \left(\|X^2\|+\|X^4\|^{\frac{1}{2}}\right),\] where $XY=YX$ and $|X^2|Y^2=(Y^*)^2|X^2|.$
For any $T\in B(\mathbb{H})$, we observe that $\left\| |T|^2+|T^*|^2\right\|+2\left\|T^2\right\|\leq (\|T\|+\|T^2\|^{\frac{1}{2}})^2$ (see Remark \ref{remark-p}) and so the inequality obtained in Corollary \ref{corollary-4} is better than that obtained in \cite[Cor. 3.3]{A}.  
\end{remark}

\section{\textbf{Bounds for the numerical radius of operator matrices}}
\noindent In this section, we obtain some inequalities for the numerical radius of $n\times n$ operator matrices. First we prove the following theorem.

\begin{theorem}\label{th-gh}
Let $T=(T_{ij})$ be an $n\times n$ operator matrix, where $T_{ij}\in B(\mathbb{H}).$ Let $f$ and $g$ be two non-negative continuous functions on $[0,\infty)$ such that $f(t)g(t)=t,$ for all  $ t\in [0,\infty).$  Then
\[w(T)\leq w(T'),\]
where $T'=(t'_{ij})_{n\times n}$ and $t'_{ij}=\|f^2(|T_{ij}|)\|^{\frac{1}{2}} \|g^2(|T^*_{ij}|)\|^{\frac{1}{2}}$.
\end{theorem}

\begin{proof}
Let $x=(x_1,x_2,\ldots,x_n)^t\in \oplus^n_{i=1}\mathbb{H}$ with $\|x\|=1$ and  $\tilde{x}=(\|x_1\|,\|x_2\|,\ldots,\|x_n\|)^t$. Then clearly $\tilde{x}$ is a unit vector in $\mathbb{C}^n$.
Now using Lemma \ref{lemma-3}, we get
\begin{eqnarray*}
|\langle Tx,x\rangle|&=& |\sum_{i,j=1}^n\langle T_{ij}x_j,x_i\rangle|\\
&\leq& \sum_{i,j=1}^n|\langle T_{ij}x_j,x_i\rangle|\\
&\leq & \sum_{i,j=1}^n \|f(|T_{ij}|)x_j\| \|g(|T^*_{ij}|x_i)\|\\
&=& \sum_{i,j=1}^n\langle f^2(|T_{ij}|)x_j,x_j\rangle^{\frac{1}{2}} \langle g^2(|T^*_{ij}|)x_i,x_i\rangle^{\frac{1}{2}}\\
&\leq & \sum_{i,j=1}^n \|f^2(|T_{ij}|)\|^{\frac{1}{2}} \|g^2(|T^*_{ij}|)\|^{\frac{1}{2}}\|x_i\|\|x_j\|\\
&=& \langle T' \tilde{x},\tilde{x}\rangle\\
&\leq& w(T').
\end{eqnarray*}
Taking supremum over $x\in \oplus^n_{i=1}\mathbb{H},$ $ \|x\|=1$, we get the required inequality. This completes the proof.
\end{proof}

\begin{remark}
In particular, if we take $f(t)=g(t)=t^{\frac{1}{2}}$  in Theorem \ref{th-gh}, then we get, \[w(T)\leq w(\|T_{ij}\|).\] This inequality was also obtained by How and Du in \cite{HD}.
\end{remark}

Next we obtain the following inequality. 

\begin{theorem}\label{th-gk1}
Let $T=(T_{ij})$ be an $n\times n$ operator matrix, where $T_{ij}\in B(\mathbb{H}).$ Let $f$ and $g$ be two non-negative continuous functions on $[0,\infty)$ such that $f(t)g(t)=t,$ for all $ t\in [0,\infty)$. Let $T''=(t''_{ij})_{n\times n}$ where
$$ t''_{ij}= \begin{cases}
\frac{1}{2}\left(\|f^2(|T_{ij}|)+g^2(|T^*_{ij}|)\| \right), &  ~~ i=j \\
\|f^2(|T_{ij}|)\|^{\frac{1}{2}} \|g^2(|T^*_{ij}|)\|^{\frac{1}{2}}, &  ~~ i \neq j.
\end{cases}
$$ 
Then \,  $w(T)\leq w(T'').$

\end{theorem}

\begin{proof}
Let $x=(x_1,x_2,\ldots,x_n)^t\in \oplus^n_{i=1}\mathbb{H}$ with $\|x\|=1$ and  $\tilde{x}=(\|x_1\|,\|x_2\|,\ldots,\|x_n\|)^t$. Then clearly $\tilde{x}$ is a unit vector in $\mathbb{C}^n$.
Now using Lemma \ref{lemma-3}, we get
\begin{eqnarray*}
|\langle Tx,x\rangle|&=& |\sum_{i,j=1}^n\langle T_{ij}x_j,x_i\rangle|\\
&\leq& \sum_{i,j=1}^n|\langle T_{ij}x_j,x_i\rangle|\\
&\leq & \sum_{i,j=1}^n \|f(|T_{ij}|)x_j\| \|g(|T^*_{ij}|x_i)\|\\
&=& \sum_{i,j=1}^n\langle f^2(|T_{ij}|)x_j,x_j\rangle^{\frac{1}{2}} \langle g^2(|T^*_{ij}|)x_i,x_i\rangle^{\frac{1}{2}}\\
&\leq&\sum_{i=1}^n\frac{1}{2}\left(\|f^2(|T_{ii}|)+g^2(|T^*_{ii}|)\| \right)\|x_i\|^2 \\
&& \hspace{3cm}+ \sum_{i,j=1,i\neq j}^n\langle f^2(|T_{ij}|)x_j,x_j\rangle^{\frac{1}{2}} \langle g^2(|T^*_{ij}|)x_i,x_i\rangle^{\frac{1}{2}}\\
&\leq & \sum_{i=1}^n\frac{1}{2}\left(\|f^2(|T_{ii}|)+g^2(|T^*_{ii}|)\| \right)\|x_i\|^2 \\ 
&& \hspace{3cm}+\sum_{i,j=1, i\neq j}^n \|f^2(|T_{ij}|)\|^{\frac{1}{2}} \|g^2(|T^*_{ij}|)\|^{\frac{1}{2}}\|x_i\|\|x_j\|\\
&=& \langle T'' \tilde{x},\tilde{x}\rangle\\
&\leq& w(T'').
\end{eqnarray*}
Taking supremum over $x\in \oplus^n_{i=1}\mathbb{H},$ $ \|x\|=1$, we get the required inequality. This completes the proof.
\end{proof}

\begin{remark}
 \textbf{1.}  Alomari \cite[Th. 4.1]{A} also tried to improve on the inequality \cite[Th. 1]{AK}, obtained by Abu-Omar and Kittaneh. The statement of Theorem $4.1$ is not correct as it is not defined for odd $n$, e.g., if $n=5$ then we can't define $a_{33}$. Although it is defined for even $n$, the proof is incorrect and so is Corollary $4.1$. Considering $A_{11}=A_{22}=\left(\begin{array}{cc}
1&0\\
0&1
\end{array}\right)$, $A_{12}=\left(\begin{array}{cc}
0&1\\
0&0
\end{array}\right)$,  $A_{21}=\left(\begin{array}{cc}
0&0\\
1&0
\end{array}\right)$, we see that $w\left(\begin{array}{cc}
A_{11}&A_{12}\\
A_{21}&A_{22}
\end{array}\right)=2$,  whereas the upper bound obtained in \cite[Cor. 4.1]{A} is $\frac{3}{2}$. The main reason for that is author uses $|\langle A_{ij}x_j,x_i\rangle|\leq w(A_{ij})\|x_i\|\|x_j\|.$ But this is not always correct. As for example, if $A_{ij}=\left(\begin{array}{cc}
0&1\\
0&0
\end{array}\right)$, $x_j=(0,\frac{1}{\sqrt{2}})^t$ and $x_i=(\frac{1}{\sqrt{2}},0)^t$, then we see that $|\langle A_{ij}x_j,x_i\rangle|=\frac{1}{2}$ and $w(A_{ij})\|x_i\|\|x_j\|=\frac{1}{4}$.\\ 

\textbf{2.}  When this paper was uploaded on arXiv, Prof. Alomari communicated to us through e-mail and I quote a part of the communication ``Indeed, I noted on my ResearchGate profile that Theorem 4.1 is incorrect. In fact, I uploaded the wrong version of my accepted manuscript in the final step.'' However,  our mathematical comments are based on the published version of the paper that appeared in LAMA.

\end{remark}

\section{\textbf{Bounds for the B-numerical radius of operator matrices}}
\noindent In this section, we obtain bounds for the B-numerical radius of $n \times n$ operator matrices, where $B$ is an $n\times n$  diagonal operator matrix whose each diagonal entry ia a positive operator $A$ defined on $\mathbb{H}.$ To do so we need the following two lemmas, the proof of which can be found in \cite{BPN}.

\begin{lemma}$($\cite[Lemma 2.4]{BPN}$)$.\label{lem-4}
Let $T_{12},T_{21}\in B_A(\mathbb{H})$, where $A> 0$. 
If $B=\left(\begin{array}{cc}
A&O\\
O& A
\end{array}\right)$ then 
\begin{eqnarray*}
(i)&& w_B\left(\begin{array}{cc}
O&T_{12}\\
T_{21}& O
\end{array}\right)=w_B\left(\begin{array}{cc}
O&T_{21}\\
T_{12}& O
\end{array}\right),\\
(ii) && w_B\left(\begin{array}{cc}
	O&T_{12}\\
	e^{i\theta}T_{21}& O
\end{array}\right) =w_B\left(\begin{array}{cc}
	O&T_{12}\\
	T_{21}& O
\end{array}\right),~~\textit{for each}~~ \theta \in \mathbb{R}.
\end{eqnarray*}
\end{lemma}

\begin{lemma}$($\cite[Lemma 2.4]{BPN}$)$.\label{lem-1}
Let $T_{11} \in B_A(\mathbb{H})$, where $A> 0$. If $B=\left(\begin{array}{cc}
A&O\\
O& A
\end{array}\right)$ then
\[w_B\left(\begin{array}{cc}
O&T_{11}\\
T_{11}&O\end{array}\right)  =w_A(T_{11}).\]
\end{lemma}

Based on the above two lemmas, we first obtain the following upper and lower bounds for the B-numerical radius of $2\times 2$ operator matrices.
\begin{theorem}\label{th-5}
Let $T_{12},T_{21}\in B_A(\mathbb{H})$,  where $A> 0$. 
If $T=\left(\begin{array}{cc}
O&T_{12}\\
T_{21}& O
\end{array}\right)$ and $B=\left(\begin{array}{cc}
A&O\\
O& A
\end{array}\right)$ then
\begin{eqnarray*}
\frac{1}{2}\max\{w_A(T_{12}+T_{21}),w_A(T_{12}-T_{21})\} &\leq& w_B(T) \\
&\leq& \frac{1}{2}\{w_A(T_{12}+T_{21})+w_A(T_{12}-T_{21})\}.
\end{eqnarray*}
\end{theorem}

\begin{proof}
First we obtain the left hand inequality. From Lemma \ref{lem-1} we have
\begin{eqnarray*}
w_A(T_{12}+T_{21}) &=& w_B\left(\begin{array}{cc}
O&T_{12}+T_{21}\\
T_{12}+T_{21}& O
\end{array}\right)\\
\Rightarrow w_A(T_{12}+T_{21})&\leq& w_B\left(\begin{array}{cc}
O&T_{12}\\
T_{21}& O
\end{array}\right)+w_B\left(\begin{array}{cc}
O&T_{21}\\
T_{12}& O
\end{array}\right)\\
\Rightarrow w_A(T_{12}+T_{21})&\leq&2 w_B\left(\begin{array}{cc}
O&T_{12}\\
T_{21}& O
\end{array}\right), ~~\mbox{using Lemma}~~ \ref{lem-4} (i)\\
\Rightarrow \frac{1}{2}w_A(T_{12}+T_{21}) &\leq & w_B\left(T \right).
\end{eqnarray*}
Replacing $T_{21}$ by $-T_{21}$ in the above inequality and using  Lemma \ref{lem-4} (ii), we get
\[\frac{1}{2}w_A(T_{12}-T_{21})\leq w_B\left(T \right).\]
Therefore, \[\frac{1}{2}\max\{w_A(T_{12}+T_{21}),w_A(T_{12}-T_{21})\} \leq w_B(T).\] This completes the proof of the left hand inequality. Next we consider a B-unitary operator, $U=\frac{1}{\sqrt{2}}\left(\begin{array}{cc}
I&-I\\
I& I
\end{array}\right)$. Then from $w_B(T)=w_B(U^{\sharp_B}TU)$, we get
\begin{eqnarray*}
w_B\left(T \right)&=& \frac{1}{2} w_B\left(\begin{array}{cc}
T_{12}+T_{21}&T_{12}-T_{21}\\
-T_{12}+T_{21}& -T_{12}-T_{21}
\end{array}\right)\\
&\leq & \frac{1}{2} w_B\left(\begin{array}{cc}
T_{12}+T_{21}&O\\
O& -T_{12}-T_{21}
\end{array}\right)\\ && +\frac{1}{2} w_B\left(\begin{array}{cc}
O&T_{12}-T_{21}\\
-T_{12}+T_{21}& O
\end{array}\right)\\
&=& \frac{1}{2}w_A(T_{12}+T_{21})+\frac{1}{2}w_A(T_{12}-T_{21}).
\end{eqnarray*}
Thus we obtain the right hand inequality and this completes the proof.
\end{proof}
Next we obtain another bound for the B-numerical radius of $2\times 2$ operator matrices. To prove this we need the following lemma.
\begin{lemma}\label{lem-2}
Let $T_{12},T_{21}\in B_A(\mathbb{H})$, where $A> 0$.
If $T=\left(\begin{array}{cc}
O&T_{12}\\
T_{21}& O
\end{array}\right)$ and $B=\left(\begin{array}{cc}
A&O\\
O& A
\end{array}\right)$ then
\[w_B(T)=\frac{1}{2}\sup_{\theta\in \mathbb{R}}\left\|e^{i\theta} T_{12}+e^{-i\theta} T^{\sharp_A}_{21}\right\|_A.\]
\end{lemma}

\begin{proof}
We know (see \cite{Z}) that 
\begin{eqnarray*}
w_B(T)&=&\sup_{\theta\in\mathbb{R}}\|\textit{Re}_B(e^{i\theta}T)\|_B\\
&=&\frac{1}{2} \sup_{\theta\in\mathbb{R}} \|e^{i\theta}T+(e^{i\theta}T)^{\sharp_B}\|_B.
\end{eqnarray*}
 Then by simple calculation we get the required result.
\end{proof}

\begin{theorem}\label{th-3}
Let $T_{12},T_{21}\in B_A(\mathbb{H})$,  where $A> 0$.
If $T=\left(\begin{array}{cc}
O&T_{12}\\
T_{21}& O
\end{array}\right)$ and $B=\left(\begin{array}{cc}
A&O\\
O& A
\end{array}\right)$ then
\begin{eqnarray*}
w_B(T) &\leq& \left[\frac{1}{16}\|S\|_A^2+\frac{1}{4}w_A^2(T_{21}T_{12})+\frac{1}{8}w_A(T_{21}T_{12}S+ST_{21}T_{12})\right]^{\frac{1}{4}},
\end{eqnarray*}
where $S=T_{12}^{\sharp_A}T_{12}+T_{21}T_{21}^{\sharp_A}.$
\end{theorem}

\begin{proof}
From Lemma \ref{lem-2}, we have
\begin{eqnarray*}
w_B(T)&=&\frac{1}{2}\sup_{\theta\in \mathbb{R}}\left\|e^{i\theta} T_{12}+e^{-i\theta} T^{\sharp_A}_{21}\right\|_A\\
&=&\frac{1}{2}\sup_{\theta\in \mathbb{R}}\left\|(e^{i\theta} T_{12}+e^{-i\theta} T^{\sharp_A}_{21})^{\sharp_A}(e^{i\theta} T_{12}+e^{-i\theta} T^{\sharp_A}_{21})\right\|_A^{\frac{1}{2}}\\
&=&\frac{1}{2}\sup_{\theta\in \mathbb{R}} \left\| S+2 Re_A(e^{2i\theta}T_{21}T_{12}) \right\|_A^{\frac{1}{2}}\\
&=&\frac{1}{2}\sup_{\theta\in \mathbb{R}} \left\|\left( S+2 Re_A(e^{2i\theta}T_{21}T_{12})\right)^2 \right\|_A^{\frac{1}{4}}\\
&=&\frac{1}{2}\sup_{\theta\in \mathbb{R}} \left\| S^2+4 \left(Re_A(e^{2i\theta}T_{21}T_{12})\right)^2+2 Re_A\left(e^{2i\theta}(T_{21}T_{12}S+ST_{21}T_{12})\right) \right\|_A^{\frac{1}{4}}\\
\Rightarrow w^4_B(T)&\leq& \frac{1}{16}\|S\|_A^2+\frac{1}{4}w_A^2(T_{21}T_{12})+\frac{1}{8}w_A(T_{21}T_{12}S+ST_{21}T_{12}).
\end{eqnarray*}
This completes the proof of the theorem.
\end{proof}
As a consequence of  Theorem \ref{th-3}, we obtain the following inequalities.
\begin{corollary}\label{cor-3}
Let $T_{12},T_{21}\in B_A(\mathbb{H})$, where $A> 0$.
If $T=\left(\begin{array}{cc}
O&T_{12}\\
T_{21}& O
\end{array}\right)$ and $B=\left(\begin{array}{cc}
A&O\\
O& A
\end{array}\right)$ then
\begin{eqnarray*}
w_B(T) &\leq& \left[\frac{1}{16}\|P\|_A^2+\frac{1}{4}w_A^2(T_{12}T_{21})+\frac{1}{8}w_A(T_{12}T_{21}P+PT_{12}T_{21})\right]^{\frac{1}{4}},
\end{eqnarray*}
where $P=T_{21}^{\sharp_A}T_{21}+T_{12}T_{12}^{\sharp_A}.$
\end{corollary}
\begin{proof}
Interchanging $T_{12}$ and $T_{21}$ in Theorem \ref{th-3} and using Lemma \ref{lem-4} (i), we get the required inequality.
\end{proof}

\begin{corollary}\label{cor-4}
Let $T_{12},T_{21}\in B_A(\mathbb{H})$, where $A> 0$.
If $T=\left(\begin{array}{cc}
O&T_{12}\\
T_{21}& O
\end{array}\right)$ and $B=\left(\begin{array}{cc}
A&O\\
O& A
\end{array}\right)$ then
\begin{eqnarray*}
(i)~~ w_B(T)&\leq& \frac{1}{2}\sqrt{\|T_{12}^{\sharp_A}T_{12}+T_{21}T_{21}^{\sharp_A}\|_A+2w_A(T_{21}T_{12})},\\
(ii)~~ w_B(T)&\leq& \frac{1}{2}\sqrt{\|T_{12}T_{12}^{\sharp_A}+T_{21}^{\sharp_A}T_{21}\|_A+2w_A(T_{12}T_{21})}.
\end{eqnarray*}
\end{corollary}
\begin{proof}
The proof of (i) and (ii)  follows easily from Theorem \ref{th-3} and Corollary \ref{cor-3} respectively, by using the inequality in \cite[Cor. 3.3]{BPN}.
\end{proof}

\begin{remark}
Taking $T_{12}=T_{21}=T$ in Corollary \ref{cor-4} and using Lemma \ref{lem-1}, we get 
\[w_A(T)\leq \frac{1}{2}\sqrt{\|TT^{\sharp_A}+T^{\sharp_A}T\|_A+2w_A(T^2)}.\]
This inequality was also obtained in \cite[Th. 2.11]{Z}.
\end{remark}

Next we obtain the following lower bounds for the B-numerical radius of $2\times 2$ operator matrices.

\begin{theorem}\label{th-4}
Let $T_{12},T_{21}\in B_A(\mathbb{H})$,  where $A> 0$.
If $T=\left(\begin{array}{cc}
O&T_{12}\\
T_{21}& O
\end{array}\right)$ and $B=\left(\begin{array}{cc}
A&O\\
O& A
\end{array}\right)$ then
\[w_B(T)\geq  \frac{1}{2}\sqrt{\|T_{12}^{\sharp_A}T_{12}+T_{21}T_{21}^{\sharp_A}\|_A+2m_A(T_{21}T_{12})}.\]
\end{theorem}

\begin{proof}
Let $x \in \mathbb{H}$ with $\|x\|_A=1$. Let $\theta\in \mathbb{R}$ be such that $e^{2i\theta}\langle T_{21}T_{12}x,x\rangle_A=|\langle T_{21}T_{12}x,x\rangle_A|.$ From Lemma \ref{lem-2} we have
\begin{eqnarray*}
w_B(T)&\geq &\frac{1}{2}\left\|e^{i\theta} T_{12}+e^{-i\theta} T^{\sharp_A}_{21}\right\|_A\\
&=&\frac{1}{2}\left\|(e^{i\theta} T_{12}+e^{-i\theta} T^{\sharp_A}_{21})^{\sharp_A}(e^{i\theta} T_{12}+e^{-i\theta} T^{\sharp_A}_{21})\right\|_A^{\frac{1}{2}}\\
&=&\frac{1}{2}\left\| T_{12}^{\sharp_A}T_{12}+T_{21}T_{21}^{\sharp_A}+2 Re_A(e^{2i\theta}T_{21}T_{12}) \right\|_A^{\frac{1}{2}}\\
&\geq&\frac{1}{2}\left| \langle(T_{12}^{\sharp_A}T_{12}+T_{21}T_{21}^{\sharp_A})x,x\rangle_A+2 \langle Re_A(e^{2i\theta}T_{21}T_{12})x,x\rangle_A \right|^{\frac{1}{2}}\\
&\geq&\frac{1}{2}\left| \langle(T_{12}^{\sharp_A}T_{12}+T_{21}T_{21}^{\sharp_A})x,x\rangle_A+2 Re( e^{2i\theta}\langle T_{21}T_{12} x,x\rangle_A) \right|^{\frac{1}{2}}\\
&=&\frac{1}{2}\left| \langle(T_{12}^{\sharp_A}T_{12}+T_{21}T_{21}^{\sharp_A})x,x\rangle_A+2 |\langle T_{21}T_{12} x,x\rangle_A| \right|^{\frac{1}{2}}\\
&\geq&\frac{1}{2}\left( \langle(T_{12}^{\sharp_A}T_{12}+T_{21}T_{21}^{\sharp_A})x,x\rangle_A+2 m_A( T_{21}T_{12} ) \right)^{\frac{1}{2}}.\\
\end{eqnarray*}
Taking supremum over $x \in \mathbb{H}, \|x\|_A=1$, we get
\[w_B(T)\geq  \frac{1}{2}\sqrt{\|T_{12}^{\sharp_A}T_{12}+T_{21}T_{21}^{\sharp_A}\|_A+2m_A(T_{21}T_{12})}.\] This completes the proof.
\end{proof}

As a consequence of  Theorem \ref{th-4}, we get the following inequality.

\begin{corollary}\label{cor-5}
Let $T_{12},T_{21}\in B_A(\mathbb{H})$, where $A> 0$.
If $T=\left(\begin{array}{cc}
O&T_{12}\\
T_{21}& O
\end{array}\right)$ and $B=\left(\begin{array}{cc}
A&O\\
O& A
\end{array}\right)$ then
\begin{eqnarray*}
w_B(T)&\geq& \frac{1}{2}\sqrt{\|T_{12}T_{12}^{\sharp_A}+T_{21}^{\sharp_A}T_{21}\|_A+2m_A(T_{12}T_{21})}.
\end{eqnarray*}
\end{corollary}
\begin{proof}
Interchanging $T_{12}$ and $T_{21}$ in Theorem \ref{th-4} and then using Lemma \ref{lem-4} (i), we get the required inequality.
\end{proof}

\begin{remark}
Taking $T_{12}=T_{21}=T$ in Corollary \ref{cor-5} and using Lemma \ref{lem-1}, we get 
\[w_A(T)\geq \frac{1}{2}\sqrt{\|TT^{\sharp_A}+T^{\sharp_A}T\|_A+2m_A(T^2)}.\]
This inequality was also obtained in \cite[Th. 2.12]{Z}.
\end{remark}

We now obtain the following upper bound for the B-numerical radius of $n\times n$ operator matrices, where $B$ is an $n\times n$ diagonal operator matrix whose each diagonal entry is a positive operator  $A$.

\begin{theorem}\label{th-1}
Let $T=(T_{ij})$ be an $n \times n$ operator matrix with $T_{ij}\in B_A(\mathbb{H})$ and $B=$ \textit{diag} $(A,A,\ldots,A)$ be an  $n \times n$ diagonal operator matrix,  where $A\geq 0$. Let $T'=(t_{ij})_{n\times n}$ where
\[t_{ij}= \begin{cases}
w_{A}(T_{ij}), &  i=j \\
\|T_{ij}\|_A, &  i\neq j.
\end{cases}
\] 
Then \, $w_B(T)\leq w(T').$

\end{theorem}	
			
\begin{proof}
Let $x=(x_1,x_2,\ldots,x_n)^t\in \oplus^n_{i=1} \mathbb{H}$ with $\|x\|_B=1$ and $\tilde{x_A}=(\|x_1\|_A,\|x_2\|_A,\ldots,\|x_n\|_A)^t$. Then clearly $\tilde{x_A}$ is a unit vector in $\mathbb{C}^n$. Now
\begin{eqnarray*}
 |\langle Tx,x\rangle_B| &=& |\sum_{i,j=1}^n\langle T_{ij}x_j,x_i\rangle_A|\\
	 &\leq& \sum_{i,j=1}^n|\langle T_{ij}x_j,x_i\rangle_A|\\
	&=& \sum_{i=1}^n|\langle T_{ii}x_i,x_i\rangle_A|+\sum_{i,j=1,i\neq j}^n|\langle T_{ij}x_j,x_i\rangle_A|\\
	&\leq& \sum_{i=1}^n w_A(T_{ii})\|x_i\|^2_A+\sum_{i,j=1,i\neq j}^n\|T_{ij}\|_A\|x_j\|_A\|x_i\|_A\\
	&=& \sum_{i,j=1}^n t_{ij}\|x_j\|_A\|x_i\|_A\\
	&=& \langle T'\tilde{x_A},\tilde{x_A}\rangle\\
	&\leq& w(T').
\end{eqnarray*}
Taking supremum over $x \in \oplus^n_{i=1} \mathbb{H}, $ $\|x\|_B=1$, we get the desired inequality.
\end{proof}

\begin{remark}
In particular, if we consider $A=I$ in Theorem \ref{th-1} then we get the inequality \cite[Th. 1]{AK}, obtained by Abu-Omar and Kittaneh.
\end{remark}

\noindent Next we state the following lemma, which can be found in \cite[p. 44]{HJ}.
\begin{lemma}\label{l-10}
Let $T=(t_{ij})\in M_n(\mathbb{C})$ such that $t_{ij}\geq 0$ for all $i,j=1,2,\ldots,n.$ Then $$w(T)=\frac{1}{2}r\left (t_{ij}+t_{ji}\right).$$
\end{lemma}

As a consequence of Theorem \ref{th-1}, we get the following inequality.

\begin{corollary}\label{cor-1}
Let $T=(T_{ij})$ be a $2 \times 2$ operator matrix with $T_{ij}\in B_A(\mathbb{H})$ and $B=$ \textit{diag}$(A,A)$ be the $2 \times 2$ diagonal operator matrix, where $A\geq 0$. Then
\[w_B(T)\leq \frac{1}{2}\left[ w_A(T_{11})+w_A(T_{22})+\sqrt{\Big(w_A(T_{11})-w_A(T_{22})\Big)^2+\Big(\|T_{12}\|_A+\|T_{21}\|_A\Big)^2}\right].\]
\end{corollary}

\begin{proof}
From Theorem \ref{th-1} we have
\begin{eqnarray*}
w_B(T)&\leq& w\left(\begin{array}{cc}
w_A(T_{11})&\|T_{12}\|_A \\
\|T_{21}\|_A& w_A(T_{22})
\end{array}\right)\\
&=& \frac{1}{2} r\left(\begin{array}{cc}
2 w_A(T_{11})&\|T_{12}\|_A+\|T_{21}\|_A \\
\|T_{12}\|_A+\|T_{21}\|_A& 2 w_A(T_{22})
\end{array}\right),~~ \mbox{using Lemma}~~ \ref{l-10}.\\
\end{eqnarray*} 
This completes the proof.
\end{proof}

Next we obtain another upper bound for the B-numerical radius of $n\times n$ operator matrices. To do so, we need the following lemma.

\begin{lemma}\label{lem-3g}
Let $T=(T_{ij})$ be an $n \times n$ operator matrix with $T_{ij}\in B_A(\mathbb{H})$ and $B=$ \textit{diag} $(A,A,\ldots,A)$ be an  $n \times n$ diagonal operator matrix,  where $A\geq 0$. Then
\[\|T\|_B\leq \Big\|\Big(\|T_{ij}\|_A\Big)\Big\|.\]
\end{lemma}
\begin{proof}
Let $x=(x_1,x_2,\ldots,x_n)^t\in \oplus^n_{i=1} \mathbb{H}$ with $\|x\|_B=1$ and $\tilde{x_A}=(\|x_1\|_A,\|x_2\|_A,\ldots,\|x_n\|_A)^t$.  Then clearly $\tilde{x_A}$ is  a unit vector in $\mathbb{C}^n$.  We have 
\begin{eqnarray*}
\|T\|_B^2 &=& \sup_{\|x\|_B=1}\|Tx\|_B^2\\
&=& \sup_{\|x\|_B=1}|\langle Tx, Tx\rangle_B|\\
&=& \sup_{\|x\|_B=1}\Big|\sum_{i=1}^n\sum_{j=1}^n\sum_{k=1}^n\langle T_{kj}x_j, T_{ki}x_i\rangle_A\Big|\\
&\leq&\sup_{\|x\|_B=1}\sum_{i=1}^n\sum_{j=1}^n\sum_{k=1}^n|\langle T_{kj}x_j, T_{ki}x_i\rangle_A|\\
&=&\sup_{\|x\|_B=1}\sum_{i=1}^n\sum_{j=1}^n\sum_{k=1}^n|\langle T_{ki}^{\sharp_A}T_{kj}x_j, x_i\rangle_A|\\
&\leq &\sup_{\|x\|_B=1}\sum_{i=1}^n\sum_{j=1}^n\sum_{k=1}^n\| T_{ki}\|_A\|T_{kj}\|_A\|x_j\|_A\| x_i\|_A\\
&= &\sup_{\|x\|_B=1}\Big\langle \big(\|T_{ij}\|_A\big)^*\big(\|T_{ij}\|_A\big)\tilde{x_A},\tilde{x_A}\Big\rangle\\
&=&\Big\|\big(\|T_{ij}\|_A\big)^*\big(\|T_{ij}\|_A\big) \Big\|=\Big \|\big(\|T_{ij}\|_A\big) \Big\|^2.
\end{eqnarray*}
This completes the proof of the lemma.
\end{proof}

\begin{theorem}\label{th-2}
Let $T=(T_{ij})$ be an $n \times n$ operator matrix with $T_{ij}\in B_A(\mathbb{H})$ and $B=$ \textit{diag} $(A,A,\ldots,A)$ be an  $n \times n$ diagonal operator matrix,  where $A> 0$. If $C=\left(\begin{array}{cc}
A&O\\
O&A
\end{array}\right)$ then
\[w_B(T)\leq w(T{''}),\]
where $T{''}=(t'_{ij})_{n\times n}$ and 
$t'_{ij}= \begin{cases}
w_{A}(T_{ij}), &   i=j \\
w_C\left(\begin{array}{cc}
O&T_{ij}\\
T_{ji}& O
\end{array}\right), &  i \neq j.
\end{cases}
$ 
\end{theorem}

\begin{proof}
Let $\theta \in \mathbb{R}.$ Then we have
\begin{eqnarray*}
\|Re(e^{i\theta}T)\|_B &=& \Big\|\Big(\frac{1}{2}(e^{i\theta}T_{ij}+e^{-i\theta}T_{ji}^{\sharp_A})\Big)\Big\|_B\\
&\leq&\Big\|\Big(\frac{1}{2}\|e^{i\theta}T_{ij}+e^{-i\theta}T_{ji}^{\sharp_A}\|_A\Big)\Big\|, ~~\textit{using Lemma}~~ \ref{lem-3g}\\
&\leq& \left \| \left(w_C\left(\begin{array}{cc}
O&T_{ij}\\
T_{ji}& O
\end{array}\right)\right) \right\|, ~~\textit{using Lemma}~~ \ref{lem-2}\\
&=&w\left ( \left(w_C\left(\begin{array}{cc}
O&T_{ij}\\
T_{ji}& O
\end{array}\right)\right) \right), ~~\textit{using Lemma}~~ \ref{lem-4}~ (i).\\
\end{eqnarray*}
Using Lemma \ref{lem-1} and taking supremum over $\theta \in \mathbb{R}$, we get the required inequality of the theorem.
\end{proof}

\begin{remark}
 In particular, if we take $A=I$ in Theorem \ref{th-2} then we obtain the inequality obtained in \cite[Th. 2]{AK}.
\end{remark}

As a consequence of Theorem \ref{th-2}, we get the following inequality.

\begin{corollary}\label{cor-2}
Let $T=(T_{ij})$ be a $2 \times 2$ operator matrix with $T_{ij}\in B_A(\mathbb{H})$ and $B=$ \textit{diag}$(A,A)$ be the $2 \times 2$ diagonal operator matrix,  where $A> 0$. Then
\[w_B(T)\leq \frac{1}{2}\left[ w_A(T_{11})+w_A(T_{22})+\sqrt{\Big(w_A(T_{11})-w_A(T_{22})\Big)^2+4w^2_B(T_0)}\right],\]
where $T_0=\left(\begin{array}{cc}
O&T_{12}\\
T_{21}& O
\end{array}\right).$
\end{corollary}

\begin{proof}
From Theorem \ref{th-2} we have
\begin{eqnarray*}
w_B(T)&\leq& w\left(\begin{array}{cc}
w_A(T_{11})& w_B(T_{0})\\
w_B(T_{0})& w_A(T_{22})
\end{array}\right)\\
&=& r\left(\begin{array}{cc}
w_A(T_{11})& w_B(T_{0}) \\
w_B(T_{0})& w_A(T_{22})
\end{array}\right),~~ \mbox{using Lemma}~~ \ref{l-10}.\\
\end{eqnarray*} 
This completes the proof.
\end{proof}

To obtain the next inequality, we need the following lemma, which is a generalization of the polarization identity. 

\begin{lemma}\label{lemma-p}
Let $x,y\in \mathbb{H}$ and $A\in B(\mathbb{H})$ be such that $A\geq0.$ Then
\[\langle x,y\rangle_A=\frac{1}{4}\|x+y\|_A^2-\frac{1}{4}\|x-y\|_A^2+\frac{i}{4}\|x+iy\|_A^2-\frac{i}{4}\|x-iy\|_A^2.\]
\end{lemma}
\begin{proof}
Proof of this lemma follows easily by  simple calculation.
\end{proof}

\begin{theorem}\label{th-p1}
Let $X,Y\in B_A(\mathbb{H})$, where $A\geq 0.$ Then
\[w_A(Y^{\sharp_A}X)\leq \frac{1}{4}\|XX^{\sharp_A}+YY^{\sharp_A}\|_A+\frac{1}{2}w_A(XY^{\sharp_A}).\]
\end{theorem}
\begin{proof}
It is well-known that 
\begin{eqnarray*}
w_A(Y^{\sharp_A}X)&=&\sup_{\theta \in \mathbb{R}}\|\textit{Re}_A(e^{i\theta}Y^{\sharp_A}X)\|_A\\
&=&\sup_{\theta \in \mathbb{R}}w_A\left(\textit{Re}_A(e^{i\theta}Y^{\sharp_A}X)\right)\\
&=&\sup_{\theta \in \mathbb{R}}\sup_{\|x\|_A=1,x\in \mathbb{H}} \left|\langle \textit{Re}_A(e^{i\theta}Y^{\sharp_A}X)x,x\rangle_A\right|.
\end{eqnarray*}

For any $x\in \mathbb{H}, \|x\|_A=1,$ we have
\begin{eqnarray*}
\left|\langle \textit{Re}_A(e^{i\theta}Y^{\sharp_A}X)x,x\rangle_A \right|&=& \left|\textit{Re} \langle e^{i\theta}Y^{\sharp_A}Xx,x\rangle_A\right|\\
&=& \left|\textit{Re} \langle e^{i\theta}Xx,Yx\rangle_A\right| \\
&=& \left|\frac{1}{4}\|e^{i\theta}Xx+Yx\|_A^2-\frac{1}{4}\|e^{i\theta}Xx-Yx\|_A^2\right|,~~\mbox{using Lemma \ref{lemma-p}} \\
&\leq& \frac{1}{4} \max \left\{\|e^{i\theta}Xx+Yx\|_A^2,\|e^{i\theta}Xx-Yx\|_A^2 \right\} \\
&\leq& \frac{1}{4} \max \left\{\|e^{i\theta}X+Y\|_A^2,\|e^{i\theta}X-Y\|_A^2 \right\},~~\mbox{as}~ \|x\|_A=1. 
\end{eqnarray*}
Now,
\begin{eqnarray*}
\frac{1}{4}\|e^{i\theta}X+Y\|_A^2&=& \frac{1}{4}\|(e^{i\theta}X+Y)(e^{-i\theta}X^{\sharp_A}+Y^{\sharp_A})\|_A\\
&=& \frac{1}{4}\|XX^{\sharp_A}+YY^{\sharp_A}+ 2 \textit{Re}_A (e^{i\theta}XY^{\sharp_A})\|_A\\
&\leq& \frac{1}{4}\|XX^{\sharp_A}+YY^{\sharp_A}\|_A+ \frac{1}{2} \|\textit{Re}_A (e^{i\theta}XY^{\sharp_A})\|_A\\
&\leq& \frac{1}{4}\|XX^{\sharp_A}+YY^{\sharp_A}\|_A+\frac{1}{2}w_A(XY^{\sharp_A}).
\end{eqnarray*}
Similarly we can show that 
\begin{eqnarray*}
\frac{1}{4}\|e^{i\theta}X-Y\|_A^2 &\leq& \frac{1}{4}\|XX^{\sharp_A}+YY^{\sharp_A}\|_A+\frac{1}{2}w_A(XY^{\sharp_A}).
\end{eqnarray*}
Therefore, 
\begin{eqnarray*}
|\langle \textit{Re}_A(e^{i\theta}Y^{\sharp_A}X)x,x\rangle_A| &\leq& \frac{1}{4}\|XX^{\sharp_A}+YY^{\sharp_A}\|_A+\frac{1}{2}w_A(XY^{\sharp_A}).
\end{eqnarray*}
Taking supremum over $x\in \mathbb{H}, \|x\|_A=1$ we get
\[w_A\left(Re_A(e^{i\theta}Y^{\sharp_A}X)\right)\leq \frac{1}{4}\|XX^{\sharp_A}+YY^{\sharp_A}\|_A+\frac{1}{2}w_A(XY^{\sharp_A}). \]
Now taking supremum over $\theta\in \mathbb{R}$, we get
\[w_A(Y^{\sharp_A}X)\leq \frac{1}{4}\|XX^{\sharp_A}+YY^{\sharp_A}\|_A+\frac{1}{2}w_A(XY^{\sharp_A}).\]
This completes the proof of the theorem.
\end{proof}

\begin{remark}
The inequality in \cite[Th. 2.10]{SMY} follows from the inequality in Theorem \ref{th-p1} by considering $A=I.$
\end{remark}

\bibliographystyle{amsplain}

\begin{thebibliography}{99}

\bibitem{ACG2} M.L. Arias, G. Corach and M.C. Gonzalez, Partial isometries in semi-Hilbertian spaces, Linear Algebra Appl. 428 (2008) 1460-1475.
	
\bibitem{AS} O.A.M.S. Ahmed, A. Saddi, A-m-Isometric operators in semi-Hilbertian spaces, Linear Algebra Appl. 436 (2012) 3930-3942.

\bibitem{AK} A. Abu-Omar and F. Kittaneh, Numerical radius inequalities for $n\times n$ operator matrices, Linear Algebra Appl. 468 (2015) 18-26.

\bibitem {A} M.W. Alomari, Refinements of some numerical radius inequalities for Hilbert space operators, Linear Multilinear Algebra, (2019), \url{https://doi.org/10.1080/03081087.2019.1624682}.

\bibitem{BFA} H. Baklouti, K. Feki and O.A.M.S. Ahmed, Joint numerical ranges of operators in semi-Hilbertian spaces, Linear Algebra Appl. 555 (2018) 266-284.

\bibitem{BBP} P. Bhunia, S. Bag and K. Paul, Numerical radius inequalities and its applications in estimation of zeros of polynomials, Linear Algebra Appl. 573 (2019) 166-177.
	
\bibitem{BBP3} P. Bhunia, S. Bag and K. Paul, Numerical radius inequalities of operator matrices with applications, Linear Multilinear Algebra, (2019), \url{https://doi.org/10.1080/03081087.2019.1634673}.

\bibitem{BPN} P. Bhunia, K. Paul and R.K. Nayak, On inequalities for A-numerical radius of operators, Electron. J. Linear Algebra, Accepted, 2020. 

\bibitem{BBP1} S. Bag, P. Bhunia  and K. Paul, Bounds of numerical radius of bounded linear operator using $t$-Aluthge transform, arXiv:1904.12096v2 [math.FA].
	
\bibitem{BBP2} P. Bhunia, S. Bag and K. Paul, Bounds for zeros of a polynomial using numerical radius, arXiv:1906.07363v1 [math.FA].

\bibitem{BPN1} P. Bhunia, K. Paul and R.K. Nayak, Numerical radius inequalities for linear operators and operator matrices, arXiv: 1908.04499v1 [math.FA].

\bibitem{B} M.L. Buzano, Generalizzatione della diseguaglianza di Cauchy-Schwarz, Rend. Sem. Mat. Univ. e Politech. Trimo 31 (1971/73) 405-409.

\bibitem{OFK} O. Hirzallah, F. Kittaneh and K. Shebrawi, Numerical radius inequalities for certain $2\times 2$ operator matrices,  Integral Equations Operator Theory 71 (2011) 129-147.

\bibitem{HD} J.C. Hou and H.K. Du, Norm inequalities of positive operator matrices, Integral Equations Operator Theory 22 (1995) 281-294.


\bibitem{HJ} R.A. Horn, C.R. Johnson, Topics in Matrix Analysis, Cambridge University Press, Cambridge, 1991.

\bibitem{K} F. Kittaneh, Numerical radius inequalities for Hilbert spaces operators, Studia Math. 168 (2005) 73-80.

\bibitem{FK3} F. Kittaneh, Commutator inequalities associated with the polar decomposition, Proc. Amer. Math. Soc. 130 (2002) 1279-1283.


\bibitem{K88} F. Kittaneh, Notes on some inequalities for Hilbert space operators, Publ. RIMS Kyoto Univ. 24 (1988) 283-293.


\bibitem {PB} K. Paul and S. Bag, On the numerical radius of a matrix and estimation of bounds for zeros of a polynomial, Int. J. Math. Math. Sci. 2012 (2012) Article Id 129132, \url{https://doi.org/10.1155/2012/129132}.

\bibitem {PB2} K. Paul and S. Bag, Estimation of bounds for the zeros of a  polynomial using numerical radius, Appl. Math. Comput. 222 (2013) 231-243. 

\bibitem {SMY} M. Sattari, M.S. Moslehian and T. Yamazaki, Some generalized numerical radius inequalities for Hilbert space operators, Linear Algebra Appl. 470 (2015) 216-227.

\bibitem{Y} T. Yamazaki, On upper and lower bounds of the numerical radius and an equality condition, Studia Math. 178 (2007) 83-89.

\bibitem{Z} A. Zamani, A-numerical radius inequalities for semi-Hilbertian space operators, Linear Algebra Appl. 578 (2019) 159-183.


\end{thebibliography}

\end{document}